\documentclass[12pt,a4paper,oneside]{amsart}%{amsbook}
\usepackage{fullpage}
\usepackage{amssymb,amsfonts,amsbsy}
\usepackage{epsfig}
\usepackage{graphicx}

\usepackage{graphics}
\usepackage{tikz}
\usetikzlibrary{decorations.pathreplacing}
\newtheorem{theorem}{Theorem}[section]
\newtheorem{lemma}[theorem]{Lemma}

\theoremstyle{definition}
\newtheorem{definition}[theorem]{Definition}

\numberwithin{equation}{section}

\DeclareMathOperator{\diam}{diam} 

\newcommand{\be}{\begin{equation}}
\newcommand{\ee}{\end{equation}}

\DeclareMathOperator{\loc}{loc}
\makeindex

\def\Xint#1{\mathchoice %for average integral
 {\XXint\displaystyle\textstyle{#1}}%
{\XXint\textstyle\scriptstyle{#1}}%
{\XXint\scriptstyle\scriptscriptstyle{#1}}%
 {\XXint\scriptscriptstyle\scriptscriptstyle{#1}}%
 \!\int}
\def\XXint#1#2#3{{\setbox0=\hbox{$#1{#2#3}{\int}$}
 \vcenter{\hbox{$#2#3$}}\kern-.5\wd0}}

 \def\dashint{\Xint-}
\begin{document}
\title{Removable sets for Orlicz-Sobolev spaces}
%\author{Nijjwal Karak$(^1)$}
\author{Nijjwal Karak}
\thanks{The author was partially supported by the Academy of Finland grant number 131477}
\address{Department of Mathematics and Statistics, University of Jyv\"askyl\"a, P.O. Box 35, FI-40014, Jyv\"askyl\"a, Finland}
\email{nijjwal.n.karak@jyu.fi}
\begin{abstract}
We study removable sets for the Orlicz-Sobolev space $W^{1,\Psi},$ for functions of the form $\Psi(t)=t^p\log^{\lambda}(e+t).$ We show that $(p,\lambda)$-porous sets lying in a hyperplane are removable and that this result is essentially sharp.
\end{abstract}
\maketitle
\section{Introduction}
In this paper, we consider removability problems for Orlicz-Sobolev spaces $W^{1,\Psi}$ with $\Psi(t)=t^p\log^{\lambda}(e+t).$ We generalize results of Koskela in \cite{Kos99} for the usual Sobolev spaces. Let us first recall some definitions. Let $\Omega$ be an open set in $\mathbb{R}^n,$ $n\geq 2.$ We say that $u$ is in the Sobolev space $W^{1,p}(\Omega)$ if $u\in L^p(\Omega),$ $1\leq p<\infty,$ and there are functions $\partial_j u\in L^p(\Omega),$ $j=1,\ldots,n,$ so that
\begin{align}\label{!}
\int_{\Omega}u\partial_j \phi\, dx= -\int_{\Omega}\phi\partial_j u\, dx
\end{align}
for each test function $\phi\in C_0^1(\Omega)$ and all $1\leq j\leq n.$ If $E\subset\mathbb{R}^n$ is a closed set of zero Lebesgue $n$-measure, then we say that $E$ is removable for $W^{1,p}$ if $W^{1,p}(\mathbb{R}^n\setminus E)=W^{1,p}(\mathbb{R}^n)$ as sets. It is not hard to check that $E$ is removable if and only if the functions $\partial_j u\in L^p(\mathbb{R}^n\setminus E)$ satisfy \eqref{!} (with $\Omega=\mathbb{R}^n$) for each $\phi\in C_0^1(\mathbb{R}^n)$ and not only for $\phi\in C_0^1(\mathbb{R}^n\setminus E).$ Similarly to the definition of $W^{1,p}(\Omega),$ $W^{1,\Psi}(\Omega)$ refers to the class of functions in $L^{\Psi}(\Omega)$ with $\partial_j u\in L^{\Psi}(\Omega),$ $j=1,2,\ldots,n.$
\begin{definition}
If $E\subset\mathbb{R}^n$ is a closed set of zero Lebesgue $n$-measure, then we say that $E$ is removable for $W^{1,\Psi}$ if $W^{1,\Psi}(\mathbb{R}^n\setminus E)=W^{1,\Psi}(\mathbb{R}^n)$ as sets.
\end{definition}
It is easy to see that removability is a local question as in the classical case. That is, $E$ is removable for $W^{1,\Psi}$ if and only if for each $x\in E$ there is $r>0$ so that $W^{1,\Psi}(B(x,r)\setminus E)=W^{1,\Psi}(B(x,r))$ as sets. Moreover, if $E\subset\Omega$ for some open set $\Omega,$ then $E$ is removable for $W^{1,\Psi}$ if and only if $W^{1,\Psi}(\Omega\setminus E)=W^{1,\Psi}(\Omega)$ as sets. Observe that, to verify the removability, it is enough to consider the functions $u\in C^1(\Omega\setminus E)\cap W^{1,\Psi}(\Omega\setminus E)$ as $W^{1,\Psi}$ is Banach space and smooth functions are dense in $W^{1,\Psi}(\Omega\ E)$ for a doubling function $\Psi.$\\
\indent In this paper, we study the removability of compact sets $E\subset\mathbb{R}^{n-1}.$ Given $1<p\leq n,$ Koskela showed in \cite{Kos99} that there are compact sets $E\subset\mathbb{R}^{n-1}\subset\mathbb{R}^n$ that are removable for $W^{1,p}(\mathbb{R}^n),$ but not for $W^{1,q}(\mathbb{R}^n)$ for any $q<p.$ This was done by introducing the class of $p$-porous sets. It is then natural to ask if a similar result holds for $W^{1,\Psi}(\mathbb{R}^n),$ for Orlicz functions $\Psi(t)=t^p\log^{\lambda}(e+t)$ in terms of $\lambda.$ We prove that this is indeed the case by studying a generalization of $p$-porosity, the $(p,\lambda)$-porosity defined in Section 4 below.\\
    
\noindent $\textbf{Theorem A.}$ Let $E\subset\mathbb{R}^{n-1}$ be compact. Let $1<p<n, \lambda\in\mathbb{R}$ or $p=1, \lambda>0$ or $p=n, \lambda\leq n-1.$ If $E$ is $(p,\lambda)$-porous, then $E$ is removable for $W^{1,\Psi}$ in $\mathbb{R}^n,$ where $\Psi(t)=t^p\log^{\lambda}(e+t).$ Moreover, for each pair $(p,\lambda)$ as above, there is a $(p,\lambda)$-porous set $E\subset\mathbb{R}^{n-1}$ that is not removable for $W^{1,\Psi'}$ for $\Psi'(t)=t^p\log^{\lambda-\epsilon}(e+t)$ for any $\epsilon>0.$\\

\indent The restrictions $\lambda>0$ for $p=1$ and $\lambda\leq n-1$ for $p=n$ are natural, see the discussion in Section 3 below.\\
\indent The main idea behind the removability of $(p,\lambda)$-porous sets is the following. As mentioned above, it suffices to prove that \eqref{!} holds for each $u\in C^1(\Omega\setminus E)\cap W^{1,\Psi}(\Omega\setminus E)$ and for each $\phi\in C_0^1(\Omega).$ By the Fubini theorem and the usual integration by parts it suffices to show that the one sided limits $\lim_{t\rightarrow 0+}u(x',t)$ and $\lim_{t\rightarrow 0-}u(x',t)$ coincide for $H^{n-1}$-a.e. $x=(x',0)\in E.$ This is established via sharp capacity estimates and the existence of \lq\lq holes\rq\rq\ in $E$ guaranteed by the porosity condition. The same idea was used also in \cite{Kos99}, but the necessary estimates and even the definition of porosity is more novel in our setting.\\ 
\indent  Similarly to \cite{Kos99}, Theorem A yields the following result on Orlicz-Poincar\'e inequalities:\\
$\textbf{Corollary.}$ Let $n\geq 2$ be an integer, $1<p<n$ and $\lambda\in\mathbb{R}.$ There is a locally compact $n$-regular metric space that supports an Orlicz $(p,\lambda)$-Poincar\'e inequality but does not support an Orlicz $(p,\lambda-\epsilon)$-Poincar\'e inequality for any $\epsilon>0.$\\
\indent The above corollary shows that there is no self-improvement in an Orlicz $(p,\lambda)$-Poincar\'e inequality in the non-complete setting (notice that $\mathbb{R}^n\setminus E$ is not complete). This partially motivates this note. For a complete $n$-regular space, an Orlicz $(p,\lambda)$-Poincar\'e inequality, $1<p<\infty,$ always improves even in $p$ when $\lambda<p-1.$ For the case $\lambda=0,$ see \cite{KZ08} and for general $\lambda,$ see \cite{Dej}.\\
\indent For the definition of an Orlicz $(p,\lambda)$-Poincar\'e inequality see Section 2 below. The definition of porosity is given in Section 4.\\
\indent In order to make this paper more readable, we organize it as follows. In Section 2 we recall definitions and preliminary results. As Theorem A admits a more elementary proof in the planar case, we begin by proving Theorem A in Section 3 in the plane. In Section 4 we describe the modifications necessary for handling the higher dimensional situation.\\
\\
\indent {\textit{Acknowledgement}.} I wish to thank my advisor Professor Pekka Koskela for suggesting the problem addressed in this paper.
\section{Notation and preliminaries}
A function $\Psi:[0,\infty)\rightarrow [0,\infty)$ is a Young function if
\begin{align*}
\Psi(s)=\int_0^s \psi(t)\, dt,
\end{align*}
where $\psi:[0,\infty)\rightarrow [0,\infty)$ with $\psi(0)=0,$ is an increasing, left-continuous function which is neither identically zero nor identically infinite on $(0,\infty).$ A Young function $\Psi$ is convex, increasing, left-continuous and satisfies
\begin{align*}
\Psi(0)=0, \ \lim_{t\rightarrow\infty}\Psi(t)=\infty . 
\end{align*}
The generalized inverse of a Young function $\Psi,$ $\Psi^{-1}:[0,\infty]\rightarrow[0,\infty],$ is defined by the formula
\begin{equation*}
\Psi^{-1}(t)=\inf\{s:\Psi(s)>t\},
\end{equation*}
where $\inf(\emptyset)=\infty.$ A Young function $\Psi$ and its generalized inverse satisfy the double inequality
\begin{equation*}
\Psi(\Psi^{-1}(t))\leq t \leq \Psi^{-1}(\Psi(t))
\end{equation*}
for all $t\geq 0.$ In this article we will only consider the Young functions $\Psi(t)=t^p\log^{\lambda}(e+t),$ $1\leq p\leq n,\ \lambda\in\mathbb{R}.$ For a general Young function $\Psi,$ the Orlicz space $L^{\Psi}(\Omega)$ is defined by
\begin{align*}
L^{\Psi}(\Omega)=\{u:\Omega\rightarrow[-\infty,\infty]: u ~\text{measurable},~ \int_{\Omega}\Psi(\alpha\vert u\vert)\, dx<\infty ~\text{for some}~ \alpha>0\}.
\end{align*}
As in the theory of $L^p$-spaces, the elements in $L^{\Psi}(\Omega)$ are actually equivalence classes consisting of functions that differ only on a set of measure zero. The Orlicz space $L^{\Psi}(\Omega)$ is a vector space and, equipped with the Luxemburg norm
\begin{align*}
\vert\vert u\vert\vert_{L^{\Psi}(\Omega)}=\inf \{k>0: \int_{\Omega}\Psi\left(\frac{\vert u\vert}{k}\right)\, dx\leq 1\},
\end{align*}
a Banach space, see \cite[Theorem 3.3.10]{RR91}. A function $u\in L^{\Psi}(\Omega)$ is in the Orlicz-Sobolev space $W^{1,\Psi}(\Omega)$ if its weak partial derivatives (distributional derivatives) $\partial_j u$ belong to $L^{\Psi}(\Omega)$ for all $1\leq j\leq n.$ The space $W^{1,\Psi}(\Omega)$ is a Banach space with respect to the norm
\begin{align*}
\vert\vert u\vert\vert_{W^{1,\Psi}(\Omega)}=\vert\vert u\vert\vert_{L^{\Psi}(\Omega)}+\vert\vert \nabla u\vert\vert_{L^{\Psi}(\Omega)},
\end{align*}
where $\nabla u=(\partial_1 u,\ldots,\partial_n u).$ For a proof, see for example \cite[Theorem 9.3.3]{RR91}. For more about Young functions, Orlicz spaces and Orlicz-Sobolev spaces, see e.g. \cite{Tuo04,RR91}. Recall that a Young function $\Psi:[0,\infty)\rightarrow [0,\infty)$ is said to be doubling if there is a constant $C>0,$ called a doubling constant of $\Psi,$ such that 
$$\Psi(2t)\leq C\Psi(t)$$
for each $t\geq 0.$ Sometimes the doubling condition is also called the $\Delta_2$-condition.\\
%\indent Let $u$ be a continuous function in a metric space $X$. We say that a Borel measurable %function $g\geq 0$ is an upper gradient of $u$ provided
%\begin{align*}
%\vert u(x)-u(y)\vert\leq\int_{\gamma}gdH^1
%\end{align*}
%for all $x,y\in X$ and each rectifiable curve $\gamma$ that joins $x$ and $y$. Here $H^1$ denotes %the $1$-dimensional Hausdorff measure, normalized so that $H^1([0,1])=1$. Notice that if $X$ is a %domain in $\mathbb{R}^n$ and $u\in C^1(X)$, then $g=\vert\nabla u\vert$ is an upper gradient of %$u$.\\
\indent Let us also recall the Poincar\'e and the $\Psi$-Poincar\'e inequalities. A pair $u\in L_{\loc}^1(\Omega)$ and a measurable function $g\geq 0$ satisfy a $(1,p)$-Poincar\'e inequality, $p\geq 1,$ if there exist constants $C_p>0$ and $\tau\geq 1,$ such that
\begin{align}\label{poincare}
\dashint_B\vert u-u_B\vert\, dx\leq C_pr\left(\dashint_{\tau B}g^p\, dx\right)^p
\end{align}
for each ball $B=B(x,r)$ satisfying $\tau B\subset\Omega.$ Recall that if $\Omega\subset\mathbb{R}^n$ and $u\in W_{\loc}^{1,1}(\Omega),$ then the inequality \eqref{poincare} holds for $g=\vert \nabla u\vert$ with $\tau=1,$ $p=1$ and the constant depending only on $n.$ 
%Here $u_B$ is the average of $u$ in %$B(x,r)$ and the barred integrals are the averaged integrals, that is $\dashint_A %vd\mu=\mu(A)^{-1}\int_A v\, d\mu.$ Similarly, 
Similarly, a function $u\in L_{\loc}^1(\Omega)$ and a measurable function $g\geq 0$ satisfy a $\Psi$-Poincar\'e inequality, if there exist constants $C_{\Psi}>0$ and $\tau\geq 1,$ such that
\begin{align}\label{psipoincare}
\dashint_B\vert u-u_B\vert\, dx\leq C_{\Psi}r\Psi^{-1}\left(\dashint_{\tau B}\Psi(g)\, dx\right)
\end{align}
for each ball $B=B(x,r)$ satisfying $\tau B\subset\Omega.$ 
%If the inequality \eqref{psipoincare} holds for each function $u\L_{\loc}^1$ and every upper %gradient $g$ of $u$ with fixed constants, then we say that $\Omega$ supports a $\Psi$-Poincar\'e %inequality. 
Here $u_B$ is the average of $u$ in $B(x,r)$ and the barred integrals are the averaged integrals, that is $\dashint_A v\, d\mu=\mu(A)^{-1}\int_A v\, d\mu.$
\section{The planar case}
Let $E\subset(0, 1)$ be a compact set in $\mathbb{R}\subset\mathbb{R}^2.$ We say that $E$ is $(p, \lambda)$-removable if $E$ is removable for $W^{1, \Psi}$ for the function $\Psi(t)=t^p\log^{\lambda}(e+t),$ where $p\in\left[1, \infty\right)$ and $\lambda$ is any real number. It is easy to check that $(p,\lambda)$-removability is equivalent to the requirement that for each  $u\in W^{1, \Psi}(B(0,2)\setminus E)\cap C^1(B(0,2)\setminus E),$ $u^+(x)=u^-(x)$ holds for $H^1$-a.e. $x\in E.$ Here $u^+(x)=\lim_{t\rightarrow 0+}u(x_1,t),$ $u^-(x)=\lim_{t\rightarrow 0-}u(x_1,t)$ and these limits exist for $H^1$-a.e. $x=(x_1,0)\in E,$ by the Fubini theorem and the fundamental theorem of calculus. Removability of a set $E$ may depend on the exponents $p$ and $\lambda.$ Indeed, when $p>2,~\lambda\in\mathbb{R}$ and $p=2,~\lambda>1$ the complementary intervals of $E$ in $(0, 1)$ play no role for the removability, since in this case any totally disconnected closed set $E\subset (0, 1)$ is removable for $W^{1,\Psi}$ (see \cite[prop.2.1]{Kos99}, \cite[sec.9.3]{RR91} and \cite[sec.2]{Ada77}). The point here is that, for these values of $p,$ $\lambda,$ one has $u^+(x)=u^-(x)$ for all $x=(x_1,0).$\\
\indent The idea behind our definition of porosity and its applicability is the the following. If a continuous function $u\in W^{1,\Psi}$ equals one on $I_1$ and zero on $I_2$ in Figure \ref{fig1}, then using a chaining argument and the usual Poincar\'e inequality one can verify the capacity type estimate
\begin{equation*}
\int_{B(x,r)}\Psi(\vert\nabla u\vert)\geq cs^{2-p}\log^{\lambda}\left(\frac{1}{s}\right)
\end{equation*}
for $1\leq p<2$, where $s=\diam(I_2)$ and one has a similar estimate for $p=2$ also.
\begin{figure}[h!]
\begin{center}
\begin{tikzpicture}[scale=2]
\draw (0,0) circle (1);
%node [above] at (circle.north) {$B(x,r)$};
%\coordinate [label=right:$(B(x,r))$] at (1.5,0);
\draw (1.005,1.005) node {$B(x,r)$};
\draw (0.4,0)--(0.7,0);
\draw (0,0)--(0,1);
\draw (-0.20,0.5) node {$I_1$};
\draw (0.55,0.15) node {$I_2$};
\draw (0.3,0.65)  node {$u=1$};
\draw (0.55,-0.15) node {$u=0$};
\end{tikzpicture}
\end{center}
\caption{}
\label{fig1}
\end{figure}
On the other hand, $\int_{B(x,r)}\Psi(\vert\nabla u\vert)=o(r),$ for $H^1$-a.e. $x=(x_1,0).$ This leads us to the following definition.
\begin{definition}
We say that $E\subset(0,1)$ is $(p, \lambda)$-porous, $1\leq p<2$ and $\lambda\in\mathbb{R},$ if for $H^1$-a.e. $x=(x_1, 0)\in E$ there is a sequence of numbers $r_i>0$ and a constant $C_x>0$ such that $r_i\rightarrow 0$ as $i\rightarrow \infty,$ and each interval $(x_1-r_i, x_1+r_i)$ contains an interval $I_i\subset [0, 1]\setminus E$ with $H^1(I_i)^{2-p}\log^{\lambda}(1/H^1(I_i))\geq C_xr_i.$ We say that $E$ is $(2, \lambda)$-porous if we have the same as above with $\log^{\lambda-1}(1/H^1(I_i))\geq C_xr_i$ when $\lambda<1$ and $[\log\log(1/H^1(I_i))]^{-1}\geq C_xr_i$ when $\lambda=1.$ 
\end{definition}
When $\lambda=0,$ the above porosity condition is same as that of \cite{Kos99}. Notice that for $p=1,$ only the case $\lambda>0$ is non-trivial above in the sense that there are no $(1,\lambda)$-porous sets when $\lambda<0$ and a $(1,0)$-porous set necessarily has length zero.\\
\indent We begin by showing that porous sets are removable, a part of our main theorem.
\begin{theorem}\label{0}
If $E$ is $(p, \lambda)$-porous, $1\leq p< 2$ and $\lambda\in\mathbb{R},$ then $E$ is $(p, \lambda)$-removable. This is also true for $p=2$ and $\lambda\leq 1.$
\end{theorem}
\begin{proof}
As discussed in our introduction, it suffices to consider functions $u\in W^{1,\Psi}(B(0,2)\setminus E)\cap C^1(B(0,2)\setminus E).$ First note that for all $t\geq 0$
\begin{align}\label{"}
\Psi^{-1}(t)\approx t^\frac{1}{p}/\log^{\frac{\lambda}{p}}(e+t),
\end{align}
where $\Psi^{-1}$ is the generalised inverse of $\Psi.$ Also we have, by the usual covering theorems \cite[p.118]{Zie89}, that
\begin{equation}\label{1}
\lim_{r\rightarrow 0}\frac{1}{r}\int_{B(x,r)}\Psi(\vert\nabla u\vert)\, dx=0
\end{equation}
for $H^1$-a.e. $x\in{B(0,2)}.$\\
$\textbf{Case~I.}~1<p<2.$
Fix $x \in E$ so that the upper and lower limits $u^+(x)$ and $u^-(x)$ exist and $( \ref{1} )$ holds and also the porosity condition holds for $x.$ It is enough to prove that $u^+(x)=u^-(x).$ Let us assume that $u^+(x)\neq u^-(x).$ So, by subtracting a constant, scaling and truncating $u,$ without any loss of generality we may assume that $u=1$ in $A^+=\{(x_1,t):0<t<\epsilon\}$ and $u=0$ in $A^-=\{(x_1,t):-\epsilon<t<0\}.$ Fix $r_i<\epsilon,$ $I_i$ as in the definition of porosity and write $I_i'=\{y\in I_i:u(y)\leq\frac{1}{36}\}$ and $I_i''=I_i\setminus I_i'.$ By symmetry, we may assume that $H^1(I_i')\geq\frac{1}{2}H^1(I_i).$ Fix a ball $B_0$ of radius $s_0=\frac{1}{2}H^1(I_i)$ centred on $I_i$ with $B_0\cap\mathbb{R}\subset I_i$ and another ball $B'$ of radius $\frac{1}{2}r_i$ centred on $A^+$ with $B'\subset B(x,r_i)^+.$ Here $B(x,r_i)^+$ denotes the upper half of the ball $B(x,r_i).$\\
\indent If we have $\dashint_{B_0}u \geq 60/81,$ then we consider a cube $Q_0$ whose sides are parallel to the axes and of side length $2s_0$ and which contains the ball $B_0.$ If we assume that $\dashint_{Q_0}u \leq 2/3,$ then by the Poincar\'e inequality we obtain $\int_{Q_0}\vert\nabla u(x)\vert\, dx \geq cs_0$ for some constant $c,$ which implies that $\int_{B(x,r_i)}{\vert\nabla u(x)\vert}\, dx \geq cr_i.$
%$$ on $A^+$ gives $\int_{B'}\vert\nabla u(x)\vert\, dx \geq cr_i$ (see th.$5.9$ of \cite{HK98}) %for some constant c, which implies that $\int_{B(x,r_i)}{\vert\nabla u(x)\vert}\, dx \geq cr_i.$ 
Using Jensen's inequality, one obtains $$\dashint_{B(x,r_i)}{\Psi(\vert\nabla u(x)\vert)}\, dx\geq\Psi\left(\dashint_{B(x,r_i)}{\vert\nabla u(x)\vert}\, dx\right)\geq\Psi\left(\frac{c}{r_i}\right)$$ which contradicts with \eqref{1} and concludes the theorem for this particular case. Now, we may assume that $\dashint_{Q_0}u \geq 2/3.$ Then by using the Fubini theorem and the fundamental theorem of calculus, we get $\int_{Q_0}{\vert\nabla u(x)\vert}\, dx \geq s_0/18^2$ and again using Jensen's inequality we get a contradiction with \eqref{1}. Therefore we assume that $\dashint_{B_0}u \leq 60/81.$\\
\indent If we have $\dashint_{B'}u \leq 61/81,$ then again we a consider a cube $Q'$ whose sides are parallel to the axes and of side length $r_i$ and which contains $B'.$ If we assume that $\dashint_{Q'}u \geq 64/81,$ then using the Poincar\'e inequality and Jensen's inequality we get a contradiction with \eqref{1} as above. Otherwise we use the Fubini theorem and the fundamental theorem of calculus and also Jensen's inequality at the end to conclude the theorem. So now we assume that $\dashint_{B'}u \geq 61/81.$\\
\indent We use the telescopic argument for the balls $B'$ and $B_0.$ This means that we consider a finite number of balls $B_0, B_1,\ldots, B_k=B'$ whose centres lie on the line joining the centres of $B'$ and $B_0$ with $\vert B_j\cap B_{j+1}\vert\geq\frac{1}{10}\vert B_j\vert$ and the radii increase geometrically so that they form a portion of a cone. We may assume that no point in $\mathbb{R}^2$ is contained in more than two of these balls. From the construction together with the Poincar\'e inequality and H\"older's inequality we have 
\begin{eqnarray}\label{1'}
\frac{1}{81}\leq\vert u_{B_0}-u_{B'}\vert\leq\sum_{j=0}^{k-1}\vert u_{B_j}-u_{B_{j+1}}\vert &\leq & \sum_{j=0}^{k-1}cs_j\dashint_{B_j}{\vert\nabla u(x)\vert}\, dx\\
& \leq & \sum_{j=0}^{k-1}cs_j\left(\dashint_{B_j}{\vert\nabla u(x)\vert}^p\, dx\right)^{\frac{1}{p}}\nonumber\\
& \leq & \sum_{j=0}^{k-1}cs_j^{1-\frac{2}{p}}\left(\int_{B_j}\vert\nabla u(x)\vert^p\, dx\right)^{\frac{1}{p}}\nonumber,
\end{eqnarray}
where $s_j$ is the radius of the ball $B_j$ for $j=0,1, \ldots, k-1.$\\
\indent First we consider the sub-case $\lambda\geq 0.$ For this case we split the balls $B_j$ into \lq\lq good\rq\rq\ part $B_j^g$ and \lq\lq bad\rq\rq\ part $B_j^b$ where $B_j^g=\lbrace x:\vert\nabla u(x)\vert\leq \text{diam}(B_j)^{-1/2}\rbrace$ and $B_j^b=\lbrace x:\vert\nabla u(x)\vert > \text{diam}(B_j)^{-1/2}\rbrace$ for $j=0,1,\ldots, k-1.$ Using this splitting one obtains
\begin{align}\label{2}
\frac{1}{81} &\leq \sum_{j=0}^{k-1}cs_j^{1/2}+\sum_{j=0}^{k-1}cs_j^{1-\frac{2}{p}}\log^{-\frac{\lambda}{p}}\left(e+s_j^{-1/2}\right)\left(\int_{B_j^b}\vert\nabla u(x)\vert^p\log^{\lambda}\left(e+\vert\nabla u(x)\vert\right)\right)^{\frac{1}{p}}
\\
&\leq cr_i^{1/2}+c\sum_{j=0}^{k-1}\frac{1}{s_j^{\frac{2-p}{p}}\log^{\frac{\lambda}{p}}\left(e+s_j^{-1/2}\right)}\left(\int_{B_j^b}{\Psi(\vert\nabla u(x)\vert)}\, dx\right)^{1/p}\nonumber .
\end{align}
%\begin{eqnarray}\label{2}
%\frac{1}{6} & \leq & %\sum_{j=1}^{k}cr_j^{1/2}+\sum_{j=1}^{k}cr_j\Psi^{-1}\Psi\left(\dashint_{B_j^b}{\vert\nabla %u(x)\vert}dx\right)\nonumber\\
%&\leq & cr_i^{1/2}+\sum_{j=1}^{k}cr_j\Psi^{-1}\left(\dashint_{B_j^b}{\Psi(\vert\nabla %u(x)\vert})dx\right).
%\end{eqnarray}
%Using (\ref{"}) we can write the inequality (\ref{2}) as
%$$\frac{1}{6}-cr_i^{1/2}\leq \sum_{j=1}^{k}\frac{cr_j\left(\dashint_{B_j^b}{\Psi(\vert\nabla %u(x)\vert)}dx\right)^{1/p}}{\log^\frac{\lambda}{p}\left(e+\dashint_{B_j^b}{\Psi(\vert\nabla %u(x)\vert)}dx \right)}.$$\\
%Using the Jensen's Inequality again and non-decreasing property of $\Psi$ we have
%\begin{eqnarray}\label{3}
%\dashint_{B_j^b}{\Psi(\vert\nabla u(x)\vert)}dx\geq\Psi\left(\dashint_{B_j^b}{\vert\nabla %u(x)\vert}dx\right) & \geq & \Psi\left(r_j^{-1/2}\right)\nonumber\\ & = & r_j^{-p/2}\log^{\lambda} %(e+r_j^{-1/2})
%\end{eqnarray}
%for $j=1, 2\ldots k$.
We again use the H\"older's inequality to obtain
\begin{align}\label{10}
\frac{1}{81}-cr_i^{\frac{1}{2}}\leq c\left(\sum_{j=0}^{k-1}\frac{1}{{s_j}^{\frac{2-p}{p-1}}\log^{\frac{\lambda}{p-1}}(e+s_j^{-1/2})}\right)^{1-\frac{1}{p}}\left(\sum_{j=0}^{k-1}\int_{B_j^b}{\Psi(\vert\nabla u(x)\vert)}\, dx\right)^{\frac{1}{p}}.
\end{align}
Since the radii of the balls $B_j$ are in geometric series, one obtains
\begin{align}\label{subcase1}
\int_{B(x,r_i)}{\Psi(\vert\nabla u(x)\vert)}\, dx\geq cs_0^{2-p}\log^{\lambda}\left(\frac{1}{s_0}\right)\left(\frac{1}{6}-cr_i^{1/2}\right)^p.
\end{align}
\indent For the sub-case $\lambda<0,$ we apply Jensen's inequality to the first line of (\ref{1'}) and use (\ref{"}) to get
\begin{eqnarray*}
\frac{1}{81} & \leq & \sum_{j=0}^{k-1}cs_j\Psi^{-1}\left(\dashint_{B_j}{\Psi(\vert\nabla u(x)\vert})\, dx\right)\\
& \leq & \sum_{j=0}^{k-1}\frac{cs_j\left(\dashint_{B_j}{\Psi(\vert\nabla u(x)\vert)}\, dx\right)^{1/p}}{\log^\frac{\lambda}{p}\left(e+\dashint_{B_j}{\Psi(\vert\nabla u(x)\vert)}\, dx \right)}.
\end{eqnarray*}
Let us consider the bigger ball $B=B(0,10)$ containing all the balls $B_j,$ $j=0,1,\ldots, k-1.$ Now $\int_{B_j}\Psi(\vert\nabla u\vert)\, dx \leq \int_{B\setminus E}\Psi(\vert\nabla u\vert)\, dx\leq M$ for $j=0,1,\ldots, k,$ where $M$ is a constant independent of $x$ and $r_i.$ Apply this estimate and the H\"older's inequality to the above inequality to obtain
\begin{eqnarray}\label{11}
\frac{1}{81} \leq c \left(\sum_{j=0}^{k-1}\frac{1}{{s_j}^{\frac{2-p}{p-1}}\log^{\frac{\lambda}{p-1}}(e+Ms_j^{-2})}\right)^{1-\frac{1}{p}} \left(\sum_{j=0}^{k-1}\int_{B_j}{\Psi(\vert\nabla u(x)\vert)}\, dx\right)^{1/p}
\end{eqnarray}
Consequently,
\begin{align}\label{subcase2}
\int_{B(x,r_i)}{\Psi(\vert\nabla u(x)\vert)}\, dx\geq cs_0^{2-p}\log^{\lambda}\left(\frac{1}{s_0}\right).
\end{align}
\indent Taking \eqref{subcase1} into account we conclude that \eqref{subcase2} holds both for $\lambda\geq 0$ and for $\lambda<0.$ Recalling that $s_0=\frac{1}{2}H^1(I_i)$ and using the porosity condition we get a contradiction with \eqref{1}.\\
$\textbf{Case~II.}~p=1,\lambda\in\mathbb{R}.$ If $\lambda\leq 0,$ then $E$ necessarily has vanishing length and removability is clear. For $\lambda>0,$ we proceed similarly like in the previous case to obtain from $(\ref{2})$
 $$\frac{1}{81}-cr_i^{1/2}\leq \sum_{j=0}^{k-1}\frac{c\left(\int_{B_j^b}{\Psi(\vert\nabla u(x)\vert)}\, dx\right)}{s_j\log^{\lambda}\left(e+{s_j}^{-\frac{1}{2}}\right)}.$$ Hence one gets the desired estimate as
 \begin{align*}
 \int_{B(x, r_i)} \Psi(\vert\nabla u(x)\vert)\geq cs_0\log^{\lambda}\left(\frac{1}{s_0}\right)
 \end{align*}
 and obtains the desired conclusion similarly as in Case I.\\
 $\textbf{Case~III.}~p=2, \lambda\leq 1.$ For $0<\lambda\leq 1,$ from the inequality (\ref{10}), we have the estimate
  \begin{align*}
\frac{1}{81}-cr_i^{1/2}\leq c\left(\sum_{j=0}^{k-1}\frac{1}{\log^{\lambda}(e+{s_j}^{-1/2})}\right)^{\frac{1}{2}}\left(\sum_{j=0}^{k-1}\int_{B_j^b}{\Psi(\vert\nabla u(x)\vert)}\, dx\right)^{\frac{1}{2}},
\end{align*}
and for $\lambda<0,$ from the inequality (\ref{11}), we have the estimate
\begin{align*}
\frac{1}{81} \leq c \left(\sum_{j=0}^{k-1}\frac{1}{\log^{\lambda}(e+Ms_j^{-2})}\right)^{1/2} \left(\sum_{j=0}^{k-1}\int_{B_j}{\Psi(\vert\nabla u(x)\vert)}\, dx\right)^{1/2}.
\end{align*}
Hence we have
\begin{align*}
\int_{B(x,r_i)}{\Psi(\vert\nabla u(x)\vert)}\, dx\geq c\log^{\lambda-1}\left(\frac{1}{s_0}\right)
\end{align*}
for $\lambda<1$ and
\begin{align*}
\int_{B(x,r_i)}{\Psi(\vert\nabla u(x)\vert)}\, dx\geq \frac{c}{\log\log\left(\frac{1}{s_0}\right)}
\end{align*}
for $\lambda=1$ and conclude similarly as in Case I to finish the proof.
\end{proof}
The next theorem shows that $E$ cannot be removable if the complementary intervals are small. This result will help us to prove the sharpness in Theorem A. For an interval $I=(a,b)$ and a positive real number $c,$ we write $cI$ to denote the interval $(\frac{a+b}{2}-\frac{c(b-a)}{2},\frac{a+b}{2}+\frac{c(b-a)}{2}).$ For a rectangle $W,$ we define $cW$ in a similar way. 
\begin{theorem}\label{4}
Let $E\subset (0,1)$ be compact with $(0,1)\setminus E=\bigcup_{j=1}^{\infty} I_j,$ where $I_j$ are pairwise disjoint open intervals. Suppose that\\
$(i)~H^1(E)>0 \quad\text{and}\quad \sum_{j=1}^\infty H^1(I_j)^{2-p}\log^{\lambda}(1/{H^1(I_j)})<\infty,$ when $1<p<2,~\lambda \in \mathbb{R}$ or $p=1,~\lambda\geq 0;$\\
$(ii)~H^1\left((0, 1)\setminus \bigcup_{i=1}^{\infty} H^1(I_j)^{-1/2}I_j\right)>0 \quad\text{and}\quad \sum_{j=1}^\infty \log^{\lambda-1}(1/{H^1(I_j)})<\infty,$ when $p=2,~\lambda<1;$\\
$(iii)~H^1\left((0, 1)\setminus \bigcup_{i=1}^{\infty} \frac{R_j}{H^1(I_j)}I_j\right)>0 \quad\text{and}\quad \sum_{j=1}^\infty (\log\log(1/{H^1(I_j)}))^{-1}<\infty,$ when $p=2,~\lambda=1.$ Here $R_j=\exp(-\log(1/H^1(I_j))^{1/2}).$\\
Then $E$ is not $(p, \lambda)$-removable.
\end{theorem}
Notice that $(i),$ for $p=1, \lambda=0,$ shows that there are no $(1,0)$-removable compact sets $E\subset (0,1)$ of positive length.\\
\indent The idea to prove this theorem is to construct a function $u\in W^{1,\Psi}(\Omega\setminus E)$ for which the two sided limits do not coincide in a subset of $E$ of positive $H^1$-measure. 
\begin{proof}
(i) Let $\Omega=B(\frac{1}{2},\frac{1}{2}).$ We define a function $u$ in $\Omega\setminus E$ as follows:
\begin{equation*}
u(x)=
 \begin{cases}
  \text{min}\lbrace \frac{x_2}{d(x, E)}, \frac{1}{\surd 2}\rbrace & \text{if $x_2\geq0$},\\
  0& \text{if $x_2<0$},
 \end{cases}
\end{equation*}
where $x_2$ is the second coordinate of $x.$ Then $u$ is locally Lipschitz and $\vert\nabla u\vert\leq M<\infty$ almost everywhere in $\Omega\setminus\bigcup_{j=1}^\infty \Delta_j,$ where $\Delta_j$ is an isosceles right angle triangle in the upper half plane with hypotenuse $I_j.$ We also have that $\vert\nabla u(x)\vert$ is comparable with $1/d(x, E)$ when $x\in\Delta_j.$ Hence, using the fact that $E$ lies outside $\Delta_j$ for all $j$ and also using polar coordinates, we have
\begin{align*}
\int_{\Delta_j}\Psi(\vert\nabla u\vert)\, dx\leq CH^1(I_j)^{2-p}\log^{\lambda}\left(\frac{1}{H^1(I_j)}\right).
\end{align*}
Then, by the assumption of the theorem, we conclude that $u\in W^{1, \Psi}(B(0,2)\setminus E).$ But when $x=(x_1,0)\in E$ we see that $u^+(x)=1/\surd 2$ whereas $u^-(x)=0.$ It is easy to check that $u$ cannot be extended to a function in $W^{1,\Psi}(\Omega).$\\
\indent(ii) Set $\Omega=(0,1)\times (-1, 1).$ Then for every $I_j$ from our collection, we define
\begin{align*}
W_j=(H^1(I_j)^{-\frac{1}{2}}I_j)\times(-H^1(I_j)^{\frac{1}{2}}, H^1(I_j)^{\frac{1}{2}}).
\end{align*}
Given $j,$ we define, for $x\in\Omega\setminus E,$
\begin{align*}
f_j(x)=\left(\vert x-x_j \vert\log\left(1/H^1(I_j)\right)\right)^{-1}\chi_{W_j\setminus H^1(I_j)^{\frac{1}{2}}W_j}(x)
\end{align*}
and $g(x)=\max_j f_j(x),$ where $x_j$ is the centre point of $I_j.$ Then $g$ is locally bounded in $\Omega\setminus E.$ Set $y=(\frac{1}{2},\ldots,\frac{1}{2}, -1)$ and for every $x\in\Omega\setminus E$ define
\begin{equation*}
u(x)=\inf_{\gamma_x} \int_{\gamma_x}g(x)\, dH^1,
\end{equation*}
where the infimum is taken over all rectifiable curves joining $x$ and $y$ in $((0, 1)\times[-1, 1])\setminus E.$ Then $u$ is locally Lipschitz in $\Omega\setminus E$ and we get
\begin{eqnarray*}
\int_{\Omega\setminus E}\Psi(\vert\nabla u\vert) &\leq& C\sum_{j=1}^{\infty}\log^{-2}\left(\frac{1}{H^1(I_j)}\right)\int_{W_j\setminus H^1(I_j)^{\frac{1}{2}}W_j}\frac{dx}{\vert x-x_j \vert^2}\log^{\lambda}\left(e+\frac{1}{\vert x-x_j \vert}\right)\\
&\leq & C\sum_{j=1}^{\infty}\log^{\lambda-1}\left(\frac{1}{H^1(I_j)}\right)<\infty,
\end{eqnarray*}
and consequently $u\in W^{1, \Psi}(\Omega\setminus E).$ But $u\geq 1/2$ in the upper half of $\Omega$ from the construction whereas $\lim_{t\rightarrow 0-}u(x', t)=0$ for all $x'\in(0, 1)\setminus \bigcup_{i=1}^{\infty} H^1(I_j)^{-1/2}I_j,$ which has positive measure by the assumption. Hence $E$ is not removable for $u.$\\
\indent (iii) This case is very similar to the previous case. Here we take the functions 
\begin{align*}
f_j(x)=\left(\log\log\left(\frac{1}{H^1(I_j)}\right)\vert x-x_j \vert\log\left(\frac{1}{\vert x-x_j \vert}\right)\right)^{-1}\chi_{W_j\setminus \frac{H^1(I_j)}{R_j}W_j}(x),
\end{align*}
where
$$W_j=\left(\frac{R_j}{H^1(I_j)}I_j\right)\times(-R_j, R_j).$$
Then we get 
\begin{eqnarray*}
\int_{\Omega\setminus E}\Psi(\vert\nabla u\vert) &\leq & C\sum_{j=1}^{\infty}\left(\log\log\left(\frac{1}{H^1(I_j)}\right)\right)^{-2}\int_{W_j\setminus \frac{H^1(I_j)}{R_j}W_j}\frac{dx}{\vert x-x_j \vert^2 \log\left(\frac{1}{\vert x-x_j \vert}\right)}\\
&\leq & C\sum_{j=1}^{\infty}\left(\log\log\left(\frac{1}{H^1(I_j)}\right)\right)^{-2}\left(\log\log\left(\frac{1}{H^1(I_j)}\right)-\log\log\left(\frac{1}{R_j}\right)\right)\\
&= & C\sum_{j=1}^{\infty}\left(\log\log\left(\frac{1}{H^1(I_j)}\right)\right)^{-1}<\infty.
\end{eqnarray*}
\end{proof}
\begin{proof}[\textbf{Proof ~of ~Theorem ~A ~for ~n=2}] Let $1<p<2, \lambda\in\mathbb{R}$ or $p=1, \lambda>0.$ By Theorem \ref{0} and Theorem \ref{4} it suffices to construct a $(p, \lambda)$-porous Cantor set $E\subset \left[0, 1\right]$ of positive length and with $\sum_{j=1}^\infty  H^1(I_j)^{2-p}\log^{\lambda-\epsilon}\left(e+1/H^1(I_j)\right)<\infty$ for every $\epsilon>0,$ where $I_j$ are the complementary intervals of $E$ on $\left[0, 1\right].$\\
\indent  We modify the example constructed by Koskela in \cite{Kos99}. The set $E$ is obtained by the following Cantor construction. Let $0<s<\frac{1}{3}$ be a small constant to be determined momentarily. We begin by deleting an open interval of length $s2^{-\frac{2}{2-p}}$ from the middle of $\left[0, 1\right].$ We are then left with two closed intervals. We continue the process as follows: if we are left with $2^{i-1}$ closed intervals, we remove from the middle of each of those intervals an open interval of length $s2^{-\frac{2i}{2-p}}/i^{\frac{\lambda}{2-p}},$ provided $i\in M=\mathbb{N}\setminus\{2^j:j\in\mathbb{N}\},$ and if we are left with $2^{2^j-1}$ closed intervals, we remove an open interval of length $s2^{-\frac{2^j}{2-p}}/2^{\frac{j\lambda}{2-p}}.$ By induction we obtain a nested sequence of closed intervals. We define $E$ as the intersection of all these closed intervals. The total length of the removed intervals is
\begin{align*}
\sum_{i\in M} 2^{i-1}\frac{s2^{-\frac{2i}{2-p}}}{i^{\frac{\lambda}{2-p}}}+\sum_{j\in\mathbb{N}} 2^{2^j-1}\frac{s2^{-\frac{2^j}{2-p}}}{2^{\frac{j\lambda}{2-p}}}<\infty.
\end{align*}
This sum can be made strictly less than $1$ by choosing $s$ sufficiently small and so $E$ has positive length. We have constructed the set $E$ in such a way that for any $x\in E$ and $j\geq 1,$ we get a complementary interval $J_j$ of length $s2^{-\frac{2^j}{2-p}}/2^{\frac{{j\lambda}}{2-p}}$ and with $d(x, J_j)\leq 2^{-2^j}.$ Hence $(p, \lambda)$-porosity of $E$ follows. Finally, to see that $E$ is not $(p, \lambda-\epsilon)$-removable, we have to check the convergence of the sum $\sum_{j=1}^\infty H^1(I_j)^{2-p}\log^{\lambda-\epsilon}(1/{H^1(I_j)}$ for $\epsilon>0,$ which in this case turns out to be
\begin{align}\label{9}
\begin{split}
\sum_{j=1}^\infty H^1(I_j)^{2-p}\log^{\lambda-\epsilon}(1/{H^1(I_j)} &=
\sum_{i\in M} 2^i\left(\frac{s2^{-\frac{2i}{2-p}}}{i^{\frac{\lambda}{2-p}}}\right)^{2-p}\log^{\lambda-\epsilon}\left(e+\frac{2^{\frac{2i}{2-p}}}{si^{-\frac{\lambda}{2-p}}}\right)\\
&\qquad +\sum_{j\in\mathbb{N}}2^{2^j}\left(\frac{s2^{-\frac{2^j}{2-p}}}{2^{\frac{j\lambda}{2-p}}}\right)^{2-p}\log^{\lambda-\epsilon}\left(e+\frac{2^{\frac{2^j}{2-p}}}{s2^{-\frac{j\lambda}{2-p}}}\right)\\
&\leq C\sum_{i\in M}\frac{2^{-i}}{i^{\epsilon}}+\sum_{j\in\mathbb{N}}\frac{1}{2^{j\epsilon}}
\end{split}
\end{align}
and hence the sum is finite for every $\epsilon>0$ (note that the sum does not converge when $\epsilon$ is zero).\\
\indent Let $p=2,~\lambda<1.$ We remove open intervals of length $s2^{-i}\exp(-2^{\frac{i}{1-\lambda}})$ when we are left with $2^{i-1}$ closed intervals for $i\in\mathbb{N}$ and then it is easy to verify the porosity condition and also the convergence of the series $\sum_{j=1}^\infty \log^{\lambda-1-\epsilon}(1/{H^1(I_j)})$ for every $\epsilon>0.$\\
\indent Now let $p=2,~\lambda=1.$ Here we remove open intervals of length $s2^{-i}\exp(-\exp(2^{i}))$ when we are left with $2^{i-1}$ closed intervals for  $i\in\mathbb{N}.$ Then one has to check that $\sum_{j=1}^\infty (\log\log(1/{H^1(I_j)}))^{-1}=\infty$ but $\sum_{j=1}^\infty \log^{-\epsilon}(1/{H^1(I_j)})<\infty$ for every $\epsilon>0,$ which is easy to do. This completes the proof of the main theorem in the plane case.
\end{proof}

\section{The higher dimensional case}
Similarly to the case $n=2,$ we would like to consider the one sided limits and to show that they coincide. But since line segments have $p$-capacity zero for $p\leq n-1,$ one can not use the same argument as in the plane case. In \cite{Kos99}, the author has used $p$-harmonic functions to overcome this problem. We do not know how to use $\Psi$-harmonic functions in our setting. Instead of this we extend the restriction of our function to the upper (or lower) half space by reflection to the entire space and take a quasicontinuous representative of this $W^{1,1}_{\loc}$-Sobolev function to reduce the problem to the following. If a function $u\in W^{1,\Psi}$ is such that $u_{B'}\geq 61/81$ and $u\leq 1/36$ on half of $A$ (see Figure \ref{fig2}), then using a chaining argument and Poincar\'e inequality we get a lower bound
%and we prove that there exists a sequence of balls $B_j,$ $j=1,2,\ldots,$ converging to $x$ such %that $\lim_{j\rightarrow\infty}u_{B_j}$ exists and equal to $\lim_{t\rightarrow 0+}u(x_1,%\ldots,x_{n-1},t)$ for $H^{n-1}$-a.e. $x\in E.$ Then we use chaining argument and Poincar\'e %inequality to get a lower bound estimate for capacity
\begin{equation*}
\int_{B(x,r)}\Psi(\vert\nabla u\vert)\geq cs^{n-p}\log^{\lambda}\left(\frac{1}{s}\right),
\end{equation*}
for $1\leq p<n-1,$ where $s=\diam(A)$ and similar estimates hold for different pairs of $(p,\lambda).$
\begin{figure}[h!]
\begin{center}
\begin{tikzpicture}[scale=2]
\draw (0,0) circle (1);
%node [above] at (circle.north) {$B(x,r)$};
%\coordinate [label=right:$(B(x,r))$] at (1.5,0);
\draw (1.005,1.005) node {$B(x,r)$};
\draw [dashed] (0.55,0) circle (0.15);
\draw (0,0.5) circle (0.3);
\draw [dashed] (0,0) -- (0,1);
\draw (-0.40,0.6) node {$B'$};
\draw (0.55,0.25) node {$A$};
\draw (-0.52,0.18) node[scale=0.7] {$u_{B'}\geq 61/81$};
\draw (0.55,-0.25) node[scale=0.7] {$u\leq 1/36$};
\end{tikzpicture}
\end{center}
\caption{}
\label{fig2}
\end{figure}
But on the other hand, we know that $\int_{B(x,r)}\Psi(\vert\nabla u\vert)=o(r^{n-1})$ for $H^{n-1}$-a.e. $x.$ Then again the definition of the porosity comes in a natural way. Before defining the porosity condition, we prove a lemma which allows us to consider even a continuum rather than a ball in the definition of porosity for some cases.
\begin{lemma}\label{3'}
Let $0<r<1.$ Denote by $B(0,r)^+$ the upper half of the $n$-dimensional ball $B(0,r)$ of radius r. Let $F\subset B(0,r)\cap\mathbb{R}^{n-1}$ be compact with $H^1(F)\geq r/3.$ Let $\mathcal{W}$ be a Whitney decomposition of $B(0,r)^+.$ Suppose $u\in C^1(B(0,r)^+)\cap C(B(0,r)^+\cup F)$ satisfies $u=0$ on $F$ and $\dashint_{Q_1} u \geq \frac{1}{2},$ where $Q_1\in\mathcal{W}$ is a largest cube contained in $B(0,r)^+.$ Then 
\begin{equation*}
\int_{B(0,r)^+}\vert\nabla u\vert ^p\log^{\lambda}\left(e+\vert\nabla u\vert\right)\geq
\begin{cases}
Cr^{n-p}\log^{\lambda}\left(\frac{1}{r}\right)& \text{when}~ n-1<p\leq n, ~\lambda\in \mathbb{R},\\
Cr\log^{\lambda-(n-2)}\left(\frac{1}{r}\right)& \text{when}~ p=n-1, ~\lambda >n-2.

\end{cases}
\end{equation*}
\end{lemma}
%\log^{\lambda-(n-1)}\left(\frac{1}{r}\right)& \text{when}~ p=n, ~\lambda <n-1,\\
%\left(\log\log\left(\frac{1}{r}\right)\right)^{1-n}& \text{when}~ p=n, ~\lambda =n-1.
\begin{proof}
First note that a change of variables $y=x/r$ gives the estimate
$$\int_{B(0,r)^+}\vert\nabla u(x)\vert^p\log^{\lambda}\left(e+\vert\nabla u(x)\vert\right)\, dx=r^{n-p}\int_{B(0,1)^+}\vert\nabla v(y)\vert^p\log^{\lambda}\left(e+\frac{\vert\nabla v(y)\vert}{r}\right)\, dy,$$
where $v(y)=u(ry)$ satisfies $v\in C^1(B(0,1)^+)\cap C(B(0,1)^+\cup F'),$ $v=0$ on $F'.$ Here $F'\subset B(0,1)\cap\mathbb{R}^{n-1}$ is the transformed compact set with $H^1(F')\geq 1/3.$ Denote $\mathcal{W}'$ the collection of cubes from $\mathcal{W}$ after rescaling. The function $v$ also satisfies $\dashint_{Q'} v \geq \frac{1}{2},$ where $Q'\subset B(0,1)^+$ is the corresponding transformed cube from the collection $\mathcal{W}_1.$\\
\indent Since $H^1(F')>0,$ Frostman's lemma (p.112 of \cite{Mat95}) implies that there exists a Radon measure $\mu$ supported in $F'$ so that $\mu(B(x,r))\leq r$ for all $x\in\mathbb{R}^n$ and all $r>0$ and that $\mu(F')\geq cH^1(F')\geq c/3,$ where $c$ is a positive constant depending only on $n.$\\
\indent For $x\in F',$ denote by $I_x$ the line segment joining $x$ to the centre of $Q'$ and let $\mathcal{Q}(x)$ consist of all the cubes $Q\in\mathcal{W}'$ such that $I_x$ intersects the cube $Q.$ Now we use the Poincar\'e inequality for the chain of cubes to obtain
\begin{align*}
\frac{1}{2}\leq \vert v(x)-v_{Q_1}\vert\leq C\sum_{Q\in\mathcal{Q}(x)} \ell(Q)\left(\dashint_{Q}\vert\nabla v\vert ^p\right)^{\frac{1}{p}},
\end{align*}
where $\ell(Q)$ denotes the edge length of $Q.$ We split the cubes $Q\in\mathcal{Q}(x)$ into \lq\lq good\rq\rq\ part $Q^g$ and \lq\lq bad\rq\rq\ part $Q^b$ where $Q^g=\lbrace x:\vert\nabla v(x)\vert\leq \ell(Q)^{-1/2}\rbrace$ and $Q^b=\lbrace x:\vert\nabla v(x)\vert>\ell(Q)^{-1/2}\rbrace.$ Using this splitting similarly to the inequality (\ref{2}), we rewrite the above inequality as
\begin{align*}
1\leq C\sum_{Q\in\mathcal{Q}(x)} \frac{\ell(Q)^{1-\frac{n}{p}}}{\log^{\frac{\lambda}{p}}\left(e+\frac{\ell(Q)^{-\frac{1}{2}}}{r}\right)}\left(\int_{Q}\vert\nabla v\vert ^p\log^{\lambda}\left(e+\frac{\vert\nabla v\vert}{r}\right)\right)^{\frac{1}{p}}
\end{align*}
for $\lambda>0.$ (For $\lambda<0$ we use Jensen's inequality similarly to the proof of Theorem \ref{0} to get the above inequality.)
By integrating with respect to $\mu$ and using the Fubini theorem and H\"older's inequality we get
\begin{eqnarray*}
\mu(F') & \leq & C\int_{F'} \sum_{Q\in\mathcal{Q}(x)} \frac{\ell(Q)^{1-\frac{n}{p}}}{\log^{\frac{\lambda}{p}}\left(e+\frac{\ell(Q)^{-\frac{1}{2}}}{r}\right)}\left(\int_{Q}\vert\nabla v\vert ^p\log^{\lambda}\left(e+\frac{\vert\nabla v\vert}{r}\right)\right)^{\frac{1}{p}}\, d\mu(x)\\
& \leq & C\sum_{Q\in\mathcal{W}'} \frac{\ell(Q)^{1-\frac{n}{p}}}{\log^{\frac{\lambda}{p}}\left(e+\frac{\ell(Q)^{-\frac{1}{2}}}{r}\right)}\left(\int_{Q}\vert\nabla v\vert ^p\log^{\lambda}\left(e+\frac{\vert\nabla v\vert}{r}\right)\right)^{\frac{1}{p}} \mu(S(Q))\\
& \leq & C\left(\sum_{Q\in\mathcal{W}'}\int_{Q}\vert\nabla v\vert ^p\log^{\lambda}\left(e+\frac{\vert\nabla v\vert}{r}\right)\right)^{\frac{1}{p}} \left(\sum_{Q\in\mathcal{W}'}\frac{\ell(Q)^{\frac{p-n}{p-1}}\mu(S(Q))^{\frac{p}{p-1}}}{\log^{\frac{\lambda}{p-1}}\left(e+\frac{\ell(Q)^{-\frac{1}{2}}}{r}\right)}\right)^{1-\frac{1}{p}},
\end{eqnarray*}
where $S(Q)\subset F'$ denotes the \lq\lq shadow\rq\rq\ of a cube $Q,$ i.e. those points $x\in F'$ for which $I_x\cap Q\neq \emptyset.$ Furthermore, denote by $\mathcal{W}_j$ all the cubes in the $j$th generation of Whitney cubes, i.e. $\mathcal{W}_j$ consists of the cubes $Q\in\mathcal{W}'$ of edge length between $2^{-j}$ and $2^{-(j+1)}.$ We deduce that
\begin{eqnarray*}
\mu(F')^p &\leq & C\int_{B(0,1)^+} \vert\nabla v\vert ^p\log^{\lambda}\left(e+\frac{\vert\nabla v\vert}{r}\right)\left(\sum_{j=1}^{\infty}\sum_{Q\in\mathcal{W}_j}\frac{2^{-j\frac{p-n}{p-1}}\mu(S(Q))^{\frac{p}{p-1}}}{\log^{\frac{\lambda}{p-1}}\left(e+\frac{2^{\frac{j}{2}}}{r}\right)}\right)^{p-1}\\
&\leq & C\mathcal{I}\left(\sum_{j=1}^{\infty}\frac{2^{-j\frac{p-n}{p-1}}\max\limits_{Q\in\mathcal{W}_j}\mu(S(Q))^\frac{1}{p-1}}{\log^{\frac{\lambda}{p-1}}\left(e+\frac{2^{\frac{j}{2}}}{r}\right)}\sum_{Q\in\mathcal{W}_j}\mu(S(Q))\right)^{p-1}\\
&\leq & C\mu(F')^{p-1}\mathcal{I}\left(\sum_{j=1}^{\infty}\frac{2^{-j\frac{p+1-n}{p-1}}}{\log^{\frac{\lambda}{p-1}}\left(e+\frac{2^{\frac{j}{2}}}{r}\right)}\right)^{p-1},
\end{eqnarray*}
where we have denoted the integral $\int_{B(0,1)^+} \vert\nabla v\vert ^p\log^{\lambda}\left(e+\frac{\vert\nabla v\vert}{r}\right)$ by $\mathcal{I}.$ Using the fact that $\mu(F')\geq c/3$ and estimating the sum in the right hand side of the above inequality, we have
\begin{align*}
1\leq C\mathcal{I}\left(\log^{-\frac{\lambda}{p-1}}\left(\frac{1}{r}\right)\right)^{p-1},
\end{align*}
when $n-1<p\leq n, \lambda\in\mathbb{R};$
\begin{align*}
1\leq C\mathcal{I}\left(\log^{-\frac{\lambda}{n-2}+1}\left(\frac{1}{r}\right)\right)^{n-2},
\end{align*}
when $p=n-1, \lambda >n-2;$
%\begin{align*}
%1\leq C\mathcal{I}\left(\log^{-\frac{\lambda}{n-1}+1}\left(\frac{1}{r}\right)\right)^{n-1},
%\end{align*}
%when $p=n, \lambda <n-1;$
%\begin{align*}
%1\leq C\mathcal{I}\left(\log\log\left(\frac{1}{r}\right)\right)^{n-1},
%\end{align*}
%when $p=n, \lambda =n-1;$ 
which implies that 
\begin{align*}
\mathcal{I}\geq C\log^{\lambda}\left(\frac{1}{r}\right) ~\text{or}~ C\log^{\lambda -(n-2)}\left(\frac{1}{r}\right)
\end{align*}
% ~\text{or}~ C\log^{\lambda -(n-1)}\left(\frac{1}{r}\right) ~\text{or}~ %C\left(\log\log\left(\frac{1}{r}\right)\right)^{1-n}
according to $n-1<p\leq n, \lambda\in\mathbb{R}$ or $p=n-1, \lambda >n-2.$ This proves the lemma.
% or $p=n, \lambda <n-1$ or $p=n, \lambda =n-1.$ 
\end{proof}
\begin{definition}\label{porosity}
We say that $E\subset\mathbb{R}^{n-1}$ is $(p,\lambda)$-porous, if for $H^{n-1}$-a.e. $x\in E,$ there is a sequence of $r_i>0$ and a constant $c_x>0$ such that $r_i\rightarrow 0$ as $i\rightarrow\infty$ and each $(n-1)$-dimensional ball $B(x,r_i)$ contains\\
(i) a ball $B_i\subset B(x,r_i)\setminus E$ of radius $R_i$ with $R_i^{n-p}\log^{\lambda}\left(\frac{1}{R_i}\right)\geq C_x{r_i}^{n-1}$ when $1\leq p<n-1$ and $\lambda\in\mathbb{R},$\\
(ii) a ball $B_i\subset B(x,r_i)\setminus E$ of radius $R_i$ with $R_i\log^{\lambda}\left(\frac{1}{R_i}\right)\geq C_x{r_i}^{n-1}$ when $p=n-1$ and $\lambda\leq n-2,$\\
(iii) a continuum $F_i\subset B(x,r_i)\setminus E$ of diameter $R_i$ with $R_i\log^{\lambda-(n-2)}\left(\frac{1}{R_i}\right)\geq C_x{r_i}^{n-1}$ when $p=n-1$ and $\lambda>n-2,$\\
(iv) a continuum $F_i\subset B(x,r_i)\setminus E$ of diameter $R_i$ with $R_i^{n-p}\log^{\lambda}\left(\frac{1}{R_i}\right)\geq C_x{r_i}^{n-1}$ when $n-1<p<n$ and $\lambda\in\mathbb{R},$\\
(v) a continuum $F_i\subset B(x,r_i)\setminus E$ of diameter $R_i$ with $\log^{\lambda-(n-1)}\left(\frac{1}{R_i}\right)\geq C_x{r_i}^{n-1}$ when $p=n$ and $\lambda<n-1,$\\
(vi) a continuum $F_i\subset B(x,r_i)\setminus E$ of diameter $R_i$ with $\left(\log\log\left(\frac{1}{R_i}\right)\right)^{1-n}\geq C_x{r_i}^{n-1}$ when $p=n$ and $\lambda=n-1.$\\
Again, the definition of porosity is same as in \cite{Kos99} for $\lambda=0.$
\end{definition}
Notice that we have replaced the round holes by holes of suitable diameter in some cases and there is a change in the power of the logarithmic term for different $p$ and also there is a mismatch in the power of the logarithmic term for different $\lambda$ when $p=n-1.$ Again we will ignore the case $p=1, \lambda\leq 0$ as in the planar case because of the same reason.\\
\indent To prove that porous sets are removable, we need help of the following lemma.

\begin{theorem}\label{main}
If $E$ is $(p,\lambda)$-porous, $1<p<n, \lambda\in\mathbb{R}$ or $p=1, \lambda>0,$ then $E$ is $(p,\lambda)$-removable. This also holds when $p=n, \lambda\leq n-1.$ 
\end{theorem}
\begin{proof}
Let $E\subset I^{n-1}=(0,1)^{n-1}$ and $u\in W^{1,\Psi}(B(0,2)\setminus E)\cap C^1(B(0,2)\setminus E).$ As in the planar case, it suffices to show that
\begin{equation}\label{5}
\int_{B^n(x,r_i)}\Psi(\vert\nabla u(x)\vert)\, dx\geq C_x r_i^{n-1}
\end{equation}
for all large enough $i$ whenever $x=(x_1,\ldots,x_{n-1},0)\in E$ is such that the one-sided limits do not coincide at $x$ and the porosity condition holds at $x.$ Here $B^n(x,r_i)$ is the $n$-dimensional ball corresponding to the $(n-1)$-dimensional ball $B(x,r_i)$ from the porosity condition. By symmetry and porosity we may assume that the upper limit is $1,$ the lower limit is $0$ and $u\leq \frac{1}{36}$ in a set $A\subset B_i$ with $H_{\infty}^{n-1}(A)\geq \frac{1}{2}H_{\infty}^{n-1}(B_i)$ or in a compact set $A\subset F_i$ with $H_{\infty}^{1}(A)\geq \frac{1}{3}H_{\infty}^{1}(F_i).$ \\
%\indent Fix a ball $B_0,$ the $n$-dimensional ball corresponding to $B_i,$ where $B_i$ is as in %the definition of porosity. If $\dashint_{B_0}u\geq 60/81,$ then we are done by using Poincar\'e %inequality, the Fubini theorem, the fundamental theorem of calculus and Jensen's inequality %similarly to the proof of Theorem \ref{0}. So, we may assume that $\dashint_{B_0}u\leq 60/81.$ %When we have continuum $F_i$ in the definition of porosity, we fix a ball $B_0$ of radius equal to %the diameter of $F_i$ such that $F_i\subset B_0\cap\mathbb{R}^{n-1}$ and we use Lemma \ref{3'} to %assume that $\dashint_{B_0^+}u\leq 60/81.$ Notice that the second case in Lemma \ref{3'} matches %with (iii) in the Definition \ref{porosity}.\\
%So, we can take an $n$-dimensional ball $B_0$ with $B_0\cap\mathbb{R}^{n-1}\subset B_i$ of radius %comparable to that of $B_i$ such that $\dashint_{B_0}u\geq 1/2$, where $B_i$ is as in the %definition of porosity. When we have continuum $F_i$ in the definition of porosity, we can take %the ball $B_0$ such that $F_i\subset B_0\cap\mathbb{R}^{n-1}$ and $\dashint_{B_0^+}u\geq 1/2$.\\
\indent Unlikely to the proof of Theorem \ref{0}, to get a ball $B'$ centred on $A^+=\{\left(x_1,\ldots,x_{n-1},t\right)\\:0<t<r_i\}$ with $\dashint_{B'}u\geq 61/81,$ one needs to do something else as $p$-capacity of a line segment in $\mathbb{R}^n$ is zero for $p\leq n-1.$ Towards this end, we prove that $\lim_{i\rightarrow\infty}u_{\hat{B}_i}$ exists for $\hat{B}_i=B\left((x_1,\ldots,x_{n-1},r_i/2),r_i/2\right)$ and is equal to $\lim_{t\rightarrow 0+}u\left(x_1,\ldots,x_{n-1},t\right)$ for $H^{n-1}$-a.e $x\in E.$\\
\indent Let $\epsilon>0.$ By reflection we obtain a function $v\in W^{1,1}((0,1)^n)$ which coincides with $u$ in the upper half plane. From the $1$-quasicontinuity of the precise representative of $v$ (for details see section $4.8$ of \cite{EG92}), we know that $\lim_{i\rightarrow\infty}v_{\hat{B}_i}$ exists outside a set $V$ with cap$_1(V)\leq\epsilon.$ Actually, \cite{EG92} considers balls centred at $x.$ However the usual Poincar\'e inequality gives this stronger statement outside an additional set of vanishing $H^{n-1}$-measure. Let
$$F=\{x\in E:\lim_{i\rightarrow\infty}v_{\hat{B}_i}\neq\lim_{t\rightarrow 0+}v\left(x_1,\ldots,x_{n-1},t\right)\}.$$
Since Hausdorff measure does not increase under projection, we have that $H^{n-1}(F)=0.$ Now, assuming that $i$ is large enough we can take a ball $B'\subset B^n(x,r_i)^+$ of radius $\frac{1}{2}r_i$ centred on $A^+$ with $\dashint_{B'}u\geq 61/81.$\\
\indent Fix a ball $B_0,$ the $n$-dimensional ball corresponding to $B_i,$ when in case (i) or (ii) of Definition \ref{porosity}. For the cases (iii)-(vi) in Definition \ref{porosity} we fix a ball $B_0$ of radius equal to the diameter of $F_i$ such that $F_i\subset \overline{B}_0\cap\mathbb{R}^{n-1}.$ Suppose that $\dashint_{B_0^+}u\leq 60/81.$ Then we use the telescopic argument for the two balls $B_0$ and $B'$ similarly to the proof of Theorem \ref{0} to get the lower bound
\begin{equation}\label{leq}
\int_{B^n(x,r_i)}\Psi(\vert\nabla u(x)\vert)\, dx\geq
\begin{cases}
C\log^{\lambda-(n-1)}\left(\frac{1}{R_i}\right)& ~\text{when}~p=n,~\lambda<n-1,\\
C\left(\log\log\left(\frac{1}{R_i}\right)\right)^{1-n}& ~\text{when}~p=n,~\lambda=n-1,\\
CR_i^{n-p}\log^{\lambda}\left(\frac{1}{R_i}\right)& ~\text{otherwise}.
\end{cases}
\end{equation}

\indent Suppose then that $\dashint_{B_0^+}u\geq 60/81.$ In the case of (i)-(ii) in Definition \ref{porosity} we use the Poincar\'e inequality, the Fubini theorem, the fundamental theorem of calculus and Jensen's inequality similarly to the proof of Theorem \ref{0} to get
\begin{equation*}
\dashint_{B^n(x,r_i)}\Psi(\vert\nabla u(x)\vert)\, dx\geq\Psi\left(\frac{C}{r_i}\right),
\end{equation*}
which contradicts with the fact that $\lim\limits_{r_i\rightarrow 0}\frac{1}{r_i^{n-1}}\int_{B^n(x,r_i)}\Psi(\vert\nabla u(x)\vert)\, dx=0$ for $H^{n-1}$-a.e. $x\in B(0,2).$ For (iii)-(vi) in Definition \ref{porosity} we apply Lemma \ref{3'} to $\frac{81}{120}u$ to conclude that
\begin{equation}\label{geq}
\int_{B^n(x,r_i)}\Psi(\vert\nabla u(x)\vert)\, dx\geq
\begin{cases}
CR_i^{n-p}\log^{\lambda}\left(\frac{1}{R_i}\right)&~\text{when}~n-1<p<n,~\lambda\in\mathbb{R}~\text{or}\\
& ~p=n,~\lambda\leq n-1,\\
CR_i\log^{\lambda-(n-2)}\left(\frac{1}{R_i}\right)&~\text{when}~p=n-1,~\lambda>n-2.
\end{cases}
\end{equation}
Taking the respective minimums of the two inequalities \eqref{leq} and \eqref{geq} and using the definition of porosity we get the inequality \eqref{5}. This completes the proof.  
%\indent Now, for the cases $1<p<n-1, \lambda\in\mathbb{R},$ $p=1, \lambda>0$ and $p=n-1, %\lambda\leq n-2,$ we use the telescopic argument for the two balls $B_0$ and $B'$ similarly to the %proof of Theorem \ref{0} to get the lower bound of the integral $\int_{B^n(x,r_i)}\Psi(\vert\nabla %u(x)\vert)\, dx\geq R_i^{n-p}\log^{\lambda}(1/R_i).$ Then using the porosity condition, we get %inequality (\ref{5}).\\
%\indent Finally, for the cases $p=n-1, \lambda>n-2;\ n-1<p<n, \lambda\in\mathbb{R};\ p=n, %\lambda<n-1$ and $p=n, \lambda=n-1,$ as we said above, we use Lemma \ref{3'} and the porosity %condition to get inequality (\ref{5}). Notice that Lemma \ref{3'} requires %$\dashint_{Q_1}u\geq\frac{1}{2}.$ If $\dashint_{Q_1}u\leq\frac{60}{81},$ then inequality (\ref{5}) %easily follows from $\dashint_{B'}u\geq \frac{61}{81}.$ Hence we may apply Lemma \ref{3'} to %$\frac{81}{120}u.$
\end{proof}
Next we give sufficient conditions for a set to be non-removable. 
\begin{theorem}\label{nonremovable}
Let $I^{n-1}\setminus E=\bigcup_{i=1}^{\infty}Q_i,$ $I=(0,1),$ where $Q_i$'s are pairwise disjoint open rectangles of length $r_i$ of one edge and of length $\sqrt{2}r_i^2$ of other edges in $I^{n-1}$ when $p=n-1, \lambda>n-2$ and $Q_i$'s are pairwise disjoint open cubes for the other values of the pair $(p,\lambda).$ Suppose that\\
$(i)~H^{n-1}(I^{n-1}\setminus\bigcup_{i=1}^{\infty}2Q_i)>0\quad\text{and}\quad\sum_{i=1}^{\infty}(\diam Q_i)^{n-p}\log^{\lambda}(1/\diam Q_i)<\infty,$ when $1<p<n-1, \lambda\in\mathbb{R}$ or $p=1, \lambda\geq 0$ or $n-1<p<n, \lambda\in\mathbb{R}$ or $p=n-1, \lambda\leq n-2;$\\
$(ii)~H^{n-1}(I^{n-1}\setminus\bigcup_{i=1}^{\infty}2Q_i)>0\quad\text{and}\quad\sum_{i=1}^{\infty}\diam Q_i\log^{\lambda-(n-2)}(1/\diam Q_i)<\infty,$ when $p=n-1, \lambda>n-2;$\\
$(iii)~H^{n-1}(I^{n-1}\setminus\bigcup_{i=1}^{\infty}(\diam Q_i)^{-\frac{1}{2}}Q_i)>0\quad\text{and}\quad\sum_{i=1}^{\infty}\log^{\lambda-(n-1)}(1/\diam Q_i)<\infty,$ when $p=n, \lambda<n-1;$\\
$(iv)~H^{n-1}(I^{n-1}\setminus\bigcup_{i=1}^{\infty}(\frac{R_i}{\diam Q_i}Q_i)>0\quad\text{and}\quad\sum_{i=1}^{\infty}(\log\log(1/\diam Q_i))^{1-n}<\infty,$ when $p=n, \lambda=n-1.$ Here $R_i=\exp(-\log(1/\diam Q_i)^{1/2}).$\\
Then $E$ is not $(p, \lambda)$-removable.

\end{theorem}
\begin{proof}
(i) Set $\Omega=I^{n-1}\times(-1,1).$ Define $W_i=(2Q_i)\times(-\diam Q_i,\diam Q_i)$ and $f_i(x)=(\diam Q_i)^{-1}\chi_{W_i}(x)$ for $x\in\Omega\setminus E$ and for every $i.$ For every $x\in\Omega\setminus E$ define $g(x)=\max_i f_i(x)$ and $u(x)=\inf_{\gamma_x}\int_{\gamma_x}g(x)\, dH^1,$ where the infimum is taken over all the rectifiable curves that join $x$ to $y=(\frac{1}{2},\ldots,\frac{1}{2},-1)$ in $(I^{n-1}\times[-1,1])\setminus E.$ To see that $u\in W^{1,\Psi}(\Omega\setminus E),$ observe that $$\int_{\Omega\setminus E}\Psi(\vert\nabla u\vert)\, dx\leq C\sum_{i=1}^{\infty}(\diam Q_i)^{n-p}\log^{\lambda}(1/\diam Q_i)<\infty.$$ As $u\geq 1$ in the upper half of $\Omega$ and $\lim_{t\rightarrow 0-}u(x',t)=0$ for all $x'\in I^{n-1}\setminus(\bigcup_{i=1}^{\infty}2Q_i),$ $E$ is not removable for $u.$\\
(ii) Set $\Omega=I^{n-1}\times(-1,1).$ We rotate $Q_i$ to form a cylinder of revolution in $\mathbb{R}^n$ of length $r_i$ and of radius of base $r_i^2$ and denote its axis by $J_i.$ Let $A_i$ denote the cylindrical annulus of length $r_i,$ inner radius $r_i^2$ and outer radius $r_i.$ Now construct $W_i$ by closing the two faces of the inner cylinder by half balls of radius $r_i^2$ and also two faces of the outer cylinder by half balls of radius $r_i,$ i.e., $W_i$ is a kind of two-sided Thermos flask (see Figure \ref{fig3}). For every $i,$ we define
\begin{equation*}
u_i(x)=
 \begin{cases}
   \frac{\log\left(\frac{1}{d(x,J_i)}\right)-\log\left(\frac{1}{r_i}\right)}{\log\left(\frac{1}{r_i^2}\right)-\log\left(\frac{1}{r_i}\right)}& \text{if $x\in W_i$},\\
  1& \text{if $x$ lies inside the inner rounded cylinder},\\
  0& \text{if $x$ lies outside the outer rounded cylinder}.
 \end{cases}
\end{equation*}
\begin{figure}[h!]
\begin{center}
\begin{tikzpicture}[scale=4]
\draw [dashed] (-0.25,0) ellipse (0.05cm and 0.25cm);
\draw [dashed] (0.25,0) ellipse (0.05cm and 0.25cm);
\draw (-0.25,0.25) -- (0.25,0.25);
\draw (-0.25,-0.25) -- (0.25,-0.25);
\draw (0.25,0.25) .. controls (0.25,0.25) and (0.5,0.25)..(0.5,0)
                  .. controls (0.5,-0.25) and (0.25,-0.25)..(0.25,-0.25);
\draw (-0.25,0.25) .. controls (-0.25,0.25) and (-0.5,0.25)..(-0.5,0)
                  .. controls (-0.5,-0.25) and (-0.25,-0.25)..(-0.25,-0.25);
\draw [dashed] (-0.25,0) ellipse (0.07cm and 0.5cm);
\draw [dashed] (0.25,0) ellipse (0.07cm and 0.5cm);                  
\draw (-0.25,0.5) -- (0.25,0.5);
\draw (-0.25,-0.5) -- (0.25,-0.5);
\draw (0.25,0.5) .. controls (0.25,0.5) and (0.75,0.5) .. (0.75,0)
                 .. controls (0.75,-0.5) and (0.25,-0.5) .. (0.25,-0.5);
\draw (-0.25,0.5) .. controls (-0.25,0.5) and (-0.75,0.5) .. (-0.75,0)
                 .. controls (-0.75,-0.5) and (-0.25,-0.5) .. (-0.25,-0.5);
\filldraw (0.25,0) circle (0.15pt);
\filldraw (-0.25,0) circle (0.15pt);  
\draw (1,0)--(1,0.5);  
\draw (0.95,0) -- (1.05,0);                                               
\draw (0.95,0.5) -- (1.05,0.5);
\draw (1.05,0.25) node {$r_i$};
\draw (0.85,0)--(0.85,0.25);  
\draw (0.80,0) -- (0.90,0);                                               
\draw (0.80,0.25) -- (0.90,0.25);
\draw (0.90,0.12) node {$r_i^2$};
\draw (-0.25,0.70) -- (0.25,0.70);
\draw (0.25,0.65) -- (0.25,0.75);
\draw (-0.25,0.65) -- (-0.25,0.75);
\draw (0,0.75) node {$r_i$};
\draw (0,0) node {$u_i=1$};
\draw (0.90,-0.40) node {$u_i=0$};
\draw (0.5,-0.25) node {$W_i$};
\end{tikzpicture}
\end{center}
\caption{}
\label{fig3}
\end{figure}
Notice that $u_i$ is Lipschitz for each $i.$ Define $u'(x)=\max_j u_j(x)$ and for every $x=(x_1,\ldots,x_n)\in\Omega\setminus E,$
\begin{equation*}
u(x)=
 \begin{cases}
  u'(x)& \text{if $x_n>0$},\\
  1& \text{if $x_n<0$},\\
  1& \text{if $x_n=0$ and $x\in Q_i$ for some $i$}.
 \end{cases}
\end{equation*}
Then by calculating the gradient and using polar coordinates, we obtain
\begin{eqnarray*}
\int_{\Omega\setminus E}\Psi(\vert\nabla u\vert)\, dx &=& \int_{\Omega}\vert\nabla u\vert^{n-1}\log^{\lambda}(e+\vert\nabla u\vert)\, dx\\
&\leq & C\sum_{i=1}^{\infty}\frac{r_i}{\left(\log\left(\frac{1}{r_i^2}\right)-\log\left(\frac{1}{r_i}\right)\right)^{n-1}}\int_{r_i^2}^{r_i}\frac{t^{n-2}}{t^{n-1}}\log^{\lambda}\left(\frac{1}{t}\right)\, dt\\
&\leq & C\sum_{i=1}^{\infty}r_i\log^{\lambda-(n-2)}\left(\frac{1}{r_i}\right)<\infty .
\end{eqnarray*}
Hence $u\in W^{1,\Psi}(\Omega\setminus E),$ but $u$ can not be extended as a Sobolev function in $\Omega.$\\
\indent Cases (iii) and (iv) can be proved similarly as cases (ii) and (iii) of Theorem \ref{4}, respectively.
\end{proof}
\begin{proof}[$\boldsymbol{Proof ~of ~Theorem ~A ~for ~n\geq 3}$] By Theorem \ref{main} and Section 3 it suffices to construct a $(p,\lambda)$-porous compact set $E\subset [0,1]^{n-1}=I^{n-1}$ of positive $H^{n-1}$-measure such that $E$ is not $(p,\lambda-\epsilon)$-removable for any $\epsilon>0.$\\
\indent Let first $1<p<n-1,\lambda\in\mathbb{R}$ or $p=1, \lambda>0$ or $n-1<p<n, \lambda\in\mathbb{R}$ or $p=n-1, \lambda\leq n-2.$ Set $A=(n-1)/(n-p).$ Let $l_0=1.$ We begin by deleting a cube $Q_0$ of edge length $s2^{-2A}$ from the centre of $I^{n-1},$ where $\frac{1}{4}\leq s \leq\frac{1}{2}$ can change its value in every stage. We subdivide $I^{n-1}\setminus Q_0$ into cubes of different sizes: $2^{n-1}$ of them of size $l_1=\frac{1}{2}(1-s2^{-2A})$ and the rest of size $s2^{-2A}.$ The cubes of size $s2^{-2A}$ correspond to translating the central cube along the coordinate directions. This determines the value of $s$ at this stage as we need $2^{2A}/s$ to be an odd integer. See Figure \ref{fig4}. Write $\mathcal{W}^1$ for the collection of all the cubes in the first subdivision.
\begin{figure}[h!]
\begin{center}
\begin{tikzpicture}[scale=4]
\draw (0,0) rectangle (1,1);
\draw (0,3/7) rectangle (1,4/7);
\draw (3/7,0) rectangle (4/7,1);
\filldraw [gray] (3/7,3/7) rectangle (4/7,4/7);
\draw (1/7,3/7) rectangle (2/7,4/7);
\draw (5/7,3/7) rectangle (6/7,4/7);
\draw (3/7,1/7) rectangle (4/7,2/7);
\draw (3/7,5/7) rectangle (4/7,6/7);
\draw [decorate, decoration=brace] (-0.07,0) -- (-0.07,3/7) node [midway,xshift=-0.25cm] {\footnotesize $l_1$};
\draw [decorate, decoration=brace] (-0.05,3/7) -- (-0.05,4/7) node [midway,xshift=-0.55cm] {\footnotesize $s2^{-2A}$};
%\draw (0.5,-0.20) node {$I^{2}$};
\draw (0.5,0.5) node {$Q_0$};
\end{tikzpicture}
\end{center}
\caption{}
\label{fig4}
\end{figure}
Then we delete a cube of edge length $s2^{-4A}/2^{\frac{2\lambda}{n-p}}$ from the centre of each cube $Q_{l_1}$ in $\mathcal{W}^1$ whose size is at least $\frac{1}{2}l_1.$ Write $\mathcal{W}^2$ for the cubes in $\mathcal{W}^1$ whose edge lengths are less than $\frac{1}{2}l_1$ and the cubes obtained from the subdivision of the cubes subject to the central deletion. The subdivision of such a cube results in cubes of two sizes: $2^{n-1}$ cubes of size $l_2=\frac{1}{2}(\ell(Q_{l_1})-s2^{-4A}/2^{\frac{2\lambda}{n-p}}),$ the rest of size $s2^{-4A}/2^{\frac{2\lambda}{n-p}}.$ We repeat the construction in the following way: at stage $i-1,$ for $i\in M=\mathbb{N}\setminus\{2^j:j\in\mathbb{N}\},$ we delete  a cube of size $s2^{-2iA}/i^{\frac{\lambda}{n-p}}$ from the centre of each cube in $\mathcal{W}^{i-1}$ of size at least $\frac{1}{2}l_{i-1}$ but when $i=2^j,$ $j\in\mathbb{N},$ we delete a cube of size $s2^{-2^jA}/2^{\frac{j\lambda}{n-p}}$ from the centre of each cube in $\mathcal{W}^{2^j-1}$ of size at least $\frac{1}{2}l_{2^j-1}.$ Then the set $$E=\left(\bigcap_{i\in M}\mathcal{W}^i\right)\bigcap\left(\bigcap_{j\in\mathbb{N}}\mathcal{W}^{2^j}\right)$$ is clearly $(p,\lambda)$-porous and finiteness of the following sum, which follows similarly as in (\ref{9}),
$$\sum_{i\in M}2^{i(n-1)}\left(\frac{2^{-2iA}}{i^{\frac{\lambda}{n-p}}}\right)^{n-p}\log^{\lambda-\epsilon}\left(2^{2iA}i^{\frac{\lambda}{n-p}}\right)+\sum_{j\in\mathbb{N}}2^{2^j(n-1)}\left(\frac{2^{-2^jA}}{2^{\frac{j\lambda}{n-p}}}\right)^{n-p}\log^{\lambda-\epsilon}\left(2^{2^jA}2^{\frac{j\lambda}{n-p}}\right)$$
gives the non-removability for every $\epsilon>0.$\\
\indent Let then $p=n$ and $\lambda<n-1.$ We begin by deleting a cube $Q_1$ of edge length $s2^{-1}\exp(-2^{\frac{1}{n-1-\lambda}})$ and then let $l_1=\frac{1}{2}(1-s2^{-1}\exp(-2^{\frac{1}{n-1-\lambda}})).$ In the $(i-1)$-th step, for $i\in\mathbb{N},$ we delete a cube of edge length $s2^{-i}\exp(-2^{\frac{i}{n-1-\lambda}})$ from the centre of each cube whose edge length is at least $\frac{1}{2}l_{i-1}.$ Then we take the set $E$ as the intersection of the collections of the remaining cubes as before and the rest is easy to verify.\\
\indent When $p=n$ and $\lambda=n-1,$ we delete a cube of edge length $s2^{-i}\exp(\exp(2^i))$ from the centre of each cube whose edge length is at least $\frac{1}{2}l_{i-1}.$\\
\indent Finally, let $p=n-1,\lambda>n-2.$ Again let $l_0=1.$ Here we begin by deleting a rectangle $Q_0$ of length $s2^{-2(n-1)}$ of one edge and of length $(s2^{-2(n-1)})^2$ of other edges from the centre of $I^{n-1}.$ We subdivide $I^{n-1}\setminus Q_0$ into cubes of different sizes: $2^{n-1}$ of them of size $l_1=\frac{1}{2}(1-(s2^{-2(n-1)})^2)$ and of rest of size $(s2^{-2(n-1)})^2.$ At this stage, $s$ can be chosen such that $2^{2(n-1)}/s$ is an odd integer. See Figure \ref{fig5}. Write $\mathcal{W}^1$ for the collection of all cubes in the first subdivision.
\begin{figure}[h!]
\begin{center}
\begin{tikzpicture}[scale=4]
\draw (0,0) rectangle (1,1);
\draw (2/5,12/25) rectangle (3/5,13/25);
%\draw (2/5,0) rectangle (3/5,1);
\filldraw [gray] (2/5,12/25) rectangle (3/5,13/25);
\draw (0,12/25) rectangle (1/25,13/25);
\draw (1/25,12/25) rectangle (2/25,13/25);
\draw (2/25,12/25) rectangle (3/25,13/25);
\draw (3/25,12/25) rectangle (4/25,13/25);
\draw (4/25,12/25) rectangle (5/25,13/25);
\draw (5/25,12/25) rectangle (6/25,13/25);
\draw (6/25,12/25) rectangle (7/25,13/25);
\draw (7/25,12/25) rectangle (8/25,13/25);
\draw (8/25,12/25) rectangle (9/25,13/25);
\draw (9/25,12/25) rectangle (10/25,13/25);
\draw (3/5,12/25) rectangle (16/25,13/25);
\draw (16/25,12/25) rectangle (17/25,13/25);
\draw (17/25,12/25) rectangle (18/25,13/25);
\draw (18/25,12/25) rectangle (19/25,13/25);
\draw (19/25,12/25) rectangle (20/25,13/25);
\draw (20/25,12/25) rectangle (21/25,13/25);
\draw (21/25,12/25) rectangle (22/25,13/25);
\draw (22/25,12/25) rectangle (23/25,13/25);
\draw (23/25,12/25) rectangle (24/25,13/25);
\draw (24/25,12/25) rectangle (25/25,13/25);
%\draw (4/5,12/25) rectangle (1,13/25);
\draw (12/25,0) rectangle (13/25,1/25);
\draw (12/25,1/25) rectangle (13/25,2/25);
\draw (12/25,2/25) rectangle (13/25,3/25);
\draw (12/25,3/25) rectangle (13/25,4/25);
\draw (12/25,4/25) rectangle (13/25,5/25);
\draw (12/25,5/25) rectangle (13/25,6/25);
\draw (12/25,6/25) rectangle (13/25,7/25);
\draw (12/25,7/25) rectangle (13/25,8/25);
\draw (12/25,8/25) rectangle (13/25,9/25);
\draw (12/25,9/25) rectangle (13/25,10/25);
\draw (12/25,10/25) rectangle (13/25,11/25);
\draw (12/25,11/25) rectangle (13/25,12/25);
\draw (12/25,13/25) rectangle (13/25,14/25);
\draw (12/25,14/25) rectangle (13/25,15/25);
\draw (12/25,15/25) rectangle (13/25,16/25);
\draw (12/25,16/25) rectangle (13/25,17/25);
\draw (12/25,17/25) rectangle (13/25,18/25);
\draw (12/25,18/25) rectangle (13/25,19/25);
\draw (12/25,19/25) rectangle (13/25,20/25);
\draw (12/25,20/25) rectangle (13/25,21/25);
\draw (12/25,21/25) rectangle (13/25,22/25);
\draw (12/25,22/25) rectangle (13/25,23/25);
\draw (12/25,23/25) rectangle (13/25,24/25);
\draw (12/25,24/25) rectangle (13/25,1);
%\filldraw [gray] (0.72,0.42) rectangle (0.84,0.58);
%\filldraw [gray] (0.42,0.14) rectangle (0.58,0.28);
%\filldraw [gray] (0.42,0.72) rectangle (0.58,0.84);
\draw [decorate, decoration=brace] (-0.07,0) -- (-0.07,12/25) node [midway,xshift=-0.25cm] {\footnotesize $l_1$};
\draw [decorate, decoration=brace] (2/5,26/25) -- (3/5,26/25) node [midway,yshift=0.35cm] {\footnotesize $s2^{-2(n-1)}$};
\draw (1.22,1/2) node[scale=0.8] {$(s2^{-2(n-1)})^2$};
%\draw (0.5,-0.20) node {$I^{2}$};
\draw (4/25,16/25) node {$Q_0$};
\draw [->] (6/25,16/25) arc (95:30:1/5);
\end{tikzpicture}
\end{center}
\caption{}
\label{fig5}
\end{figure}
We repeat the construction in the following manner: at stage $i-1,$ for $i\in M=\mathbb{N}\setminus\{2^j:j\in\mathbb{N}\},$ we delete a rectangle of length $s2^{-2i(n-1)}/i^{\lambda-(n-2)}$ of one edge and of length $\left(s2^{-2i(n-1)}/i^{\lambda-(n-2)}\right)^2$ of other edges from the centre of each cube in $\mathcal{W}^{i-1}$ of size at least $\frac{1}{2}l_{i-1},$ where
$$l_{i-1}=\frac{1}{2}\left(\ell(Q)-\left(s2^{-2i(n-1)}/i^{\lambda-(n-2)}\right)^2\right)$$
but when $i=2^j,$ $j\in\mathbb{N},$ delete a rectangle of length $s2^{-2^j(n-1)}/2^{j(\lambda-(n-2))}$ of one edge and of length $\left(s2^{-2^j(n-1)}/2^{j(\lambda-(n-2))}\right)^2$ of other edges from the centre of each cube in $\mathcal{W}^{2^j-1}$ of size at least $\frac{1}{2}l_{2^j-1},$ where
$$l_{2^j-1}=\frac{1}{2}\left(\ell(Q)-\left(s2^{-2^j(n-1)}/2^{j(\lambda-(n-2))}\right)^2\right).$$ Then write $$E=\left(\bigcap_{i\in M}\mathcal{W}^i\right)\bigcap\left(\bigcap_{j\in\mathbb{N}}\mathcal{W}^{2^j}\right),$$
which is our desired set. It is easy to check that the sum $\sum_{i=1}^{\infty}\diam Q_i\log^{\lambda-(n-2+\epsilon)}$ is finite for every $\epsilon>0,$ where $Q_i$ are the complementary rectangles in Theorem \ref{nonremovable}. This completes the proof of the theorem.
\end{proof}

\def\bibname{References}
\nocite {HK00,HK98}
%\nocite{HK,KO,EG,MA}
%\cite{EG}
%\begin{thebibliography}{999}
%\bibitem{HK}{\sc Heinonen, J. and Koskela, P.}, Quasiconformal maps in metric spaces with %controlled geometry,
%{\em Acta Math.} \textbf{181}(1998), 1-61.
%\bibitem{KO}{\sc Koskela, P.}, Removable sets for Sobolev spaces, {\em Ark Mat.} \textbf{37} %(1999), no.2, 291-304.
%\bibitem{EG}{\sc Evans, L.C. and R.F. Gariepy}, {\em Measure Theory and Fine Properties of %Functions,} CRC Press, 1992, %Boca Raton-New York-London-Tokyo.
%\bibitem{MA}{\sc Mattila, P.}, {\em Geometry of Sets and Measures in Euclidean Spaces,} (Cambridge %University Press, 1995)
%\bibitem{ZI}{\sc Ziemer, W. P.}, {\em Weakly Differentiable Functions,} Grad. Texts in Math. %\textbf{120}, Springer-Verlag, %New York, 1989.

%\end{thebibliography}
\bibliography{nijjwal}
\bibliographystyle{alpha}

\end{document}